\documentclass[12pt]{article}
\usepackage[titletoc,title]{appendix}
\usepackage{graphicx,xspace,colortbl}
\usepackage{amsmath,amsthm,amsfonts,amssymb}
\usepackage{color}
\usepackage{enumitem}
\usepackage{fancybox}
\usepackage{epsfig}
\usepackage{subfig}
\usepackage{pdfsync}
\usepackage[margin=1in]{geometry}
\usepackage[colorlinks,citecolor=blue,urlcolor=blue,linkcolor=blue]{hyperref}
\hypersetup{
colorlinks=true,
citecolor=blue,
urlcolor=blue,
linkcolor=blue,
pdfauthor = {Bartlomiej Dyda, Alexey Kuznetsov, Mateusz Kwasnicki},
pdfkeywords = {fractional Laplace operator, Riesz potential, Meijer G-function, hypergeometric function, Jacobi polynomial, harmonic polynomial, radial function},
pdftitle = {Fractional Laplace operator and Meijer G-function},
pdfpagemode = UseNone
}

\newcommand{\mk}[1]{#1}

\newcommand{\fourier}{\mathcal{F}}
\newcommand{\cG}{\mathcal{G}}
\newcommand{\mellin}{\mathcal{M}}
\renewcommand{\qed}{\hfill$\sqcap\kern-8.0pt\hbox{$\sqcup$}$\\}
\DeclareMathOperator{\re}{Re}
\DeclareMathOperator{\im}{Im}
\renewcommand{\r}{\mathbb{R}}
\renewcommand{\c}{\mathbb{C}}
\renewcommand{\d}{\textnormal{d}}
\renewcommand{\i}{\textnormal{i}}

\newtheorem{theorem}{Theorem}
\newtheorem{lemma}{Lemma}
\newtheorem{proposition}{Proposition}
\newtheorem{corollary}{Corollary}
\newtheorem*{bochners_relation}{Bochner's relation}
\theoremstyle{definition}
\newtheorem{definition}{Definition}

\newtheorem{remark}{Remark}

\title{Fractional Laplace operator and Meijer G-function}
\author{
{Bart{\l}omiej Dyda\footnotemark[1] \footnotemark[3]} , \;  
{Alexey Kuznetsov\footnotemark[2] \footnotemark[4]} , \;
{Mateusz Kwa{\'s}nicki\footnotemark[1] \footnotemark[5]}
}

\date{\today}

\begin{document}
\maketitle
{
\renewcommand{\thefootnote}{\fnsymbol{footnote}}
\footnotetext[1]{Faculty of Pure and Applied Mathematics, Wroc{\l}aw University of Technology, ul. Wybrze{\.z}e Wyspia{\'n}skiego 27, 50-370 Wroc{\l}aw, Poland. Email: \{bartlomiej.dyda,mateusz.kwasnicki\}@pwr.edu.pl}
\footnotetext[2]{Dept. of Mathematics and Statistics,  York University,
4700 Keele Street, Toronto, ON, M3J 1P3, Canada.   Email: kuznetsov@mathstat.yorku.ca}
\footnotetext[3]{Supported by Polish National Science Centre (NCN) grant no. 2012/07/B/ST1/03356}
\footnotetext[4]{Research supported by the Natural Sciences and Engineering Research Council of Canada}
\footnotetext[5]{Supported by Polish National Science Centre (NCN) grant no. 2011/03/D/ST1/00311}
}

\begin{abstract}
We significantly expand the number of functions whose image under the fractional Laplace operator can be computed explicitly. In particular, we show that the fractional Laplace operator maps Meijer G-functions of $|x|^2$, or generalized hypergeometric functions of $-|x|^2$, multiplied by a solid harmonic polynomial, into the same class of functions. As one important application of this result, we produce a complete system of eigenfunctions of the operator $(1-|x|^2)_+^{\alpha/2} (-\Delta)^{\alpha/2}$ with the Dirichlet boundary conditions outside of the unit ball. The latter result will be used to estimate the eigenvalues of the fractional Laplace operator in the unit ball in a companion paper \cite{DKK_2015}.  
\end{abstract}
{\vskip 0.15cm}
{\vskip 0.25cm}

\section{Introduction}

The fractional Laplace operator $(-\Delta)^{\alpha/2}$ is one of the most well-studied pseudo-differential operators. One reason for this lies in the multitude of important applications. This operator arises as a Markov generator of an isotropic stable process in $\r^d$, it is of increasing interest for the partial differential equations community and it is used in various areas of applied mathematics, see~\cite{bib:bv15,bib:ro15}. 
Another reason is that the fractional Laplace operator is also one of the simplest pseudo-differential operators, having symbol which is just a power function. In view of the simplicity and importance of this operator, it is surprising that there exist only a handful of functions for which the action of this operator can be computed explicitly. To the best of our knowledge, the following list provides a complete catalogue of known examples, up to translations, dilations \mk{and Kelvin transform} \mk{(as usual, we write $a_+$ for $\max(0, a)$)}:
\begin{enumerate}[label={\rm(\alph*)}]
\item Expressions for the \emph{harmonic measure} and the \emph{Green function} for a ball, as well as for the complement of the ball, were essentially established by M.~Riesz in~\cite{bib:r38a,bib:r38b} using Kelvin transform. For a formal derivation of the expression for the harmonic measure, see Section~IV.5 in~\cite{bib:l72}. The expression for the Green function in its most common form was obtained in~\cite{bib:bgr61}.
\item If $h$ is harmonic (that is, $\Delta h = 0$) in the unit ball and $f(x) = (1 - |x|^2)_+^{\alpha/2 - 1} h(x)$, then $(-\Delta)^{\alpha/2} f = 0$ in the unit ball, as it was proved in~\cite{bib:h94}, see also~\cite{bib:b99}.
\item The formula for $(-\Delta)^{\alpha/2} f$ for $f(x) = (1 - |x|^2)_+^\sigma$ with $\sigma = \tfrac{\alpha}{2}$ is essentially contained in~\cite{bib:r38a} and stated explicitly in~\cite{bib:g61}, while $\sigma = \tfrac{\alpha}{2} - 1$ is covered by the previous item. The case of general $\sigma$, as well as the function $g(x) = x_1 f(x)$, was studied in~\cite{bib:bik11,bib:d12}.
\item Similar results are easily found for half-space and for the complement of the unit ball by means of Kelvin transformation, see~\cite{bib:bz06,bib:l72,bib:r38a}
\item The results for functions supported in the full space $\r^d$ are more rare: a formula for $(-\Delta)^{\alpha/2} f(x)$ is known when $f(x) = e^{\i y \cdot x}$ (Fourier transform), $f(x) = |x|^{-a}$ with $a \in (0, d)$ (composition with Riesz kernels, \cite{bib:l72,Stein_book}), $f(x) = e^{-|x|^2}$ or $f(x) = (1 + |x|^2)^{-a}$ with $a = \tfrac{d + 1}{2}$ or $a = \tfrac{d - \alpha}{2} + n$, $n = 0, 1, \dots$ (see~\cite{Samko_book}). 
\end{enumerate}
Our goal in this paper is to extend the above list: We want to find more functions $f$ for which $(-\Delta)^{\alpha/2}f$ can be computed explicitly. Our main motivation for doing this comes from the study of spectral properties of the fractional Laplace operator acting on functions supported in the unit ball -- we consider this problem in a companion paper~\cite{DKK_2015}.

\mk{Our results are stated in terms of certain special functions: Meijer G-function, hypergeometric function and Jacobi polynomials. Before we introduce these notions formally, we summarise our findings. Our most general result (Theorems~\ref{thm:IVG} and~\ref{thm:DVG}) states that, under a number of conditions on the parameters $\alpha$, $l$, $m$, $n$, $p$, $q$, $\mathbf{a} = (a_1, \dots, a_p)$ $\mathbf{b} = (b_1, \dots, b_q)$ and the variable $x \in \r^d$, if
\begin{equation}\label{eq:VG}
f(x):=V(x) G^{mn}_{pq}\Big( 
\,\begin{matrix}
\mathbf{a} \\ \mathbf{b}
\end{matrix}\, \Big  \vert \, |x|^2 
\Big),
\end{equation}
then
\begin{equation}\label{eq:DVG}
(-\Delta)^{\alpha/2} f(x)=2^{\alpha} V(x) G^{m+1,n+1}_{p+2,q+2}\Big( 
\,\begin{matrix}
1-\tfrac{d + 2l + \alpha}{2}, & \mathbf{a}-\tfrac{\alpha}{2}, & -\tfrac{\alpha}{2} \\ 0, & \mathbf{b}-\tfrac{\alpha}{2}, &1-\tfrac{d + 2l}{2}
\end{matrix}\, \Big  \vert \, |x|^2 
\Big ) .
\end{equation}
Here $V$ is a solid harmonic polynomial of degree $l$ (one can take, for example, $V(x) = 1$ and $l = 0$, or $V(x) = x_1$ and $l = 1$) and $G^{mn}_{pq}(\mathbf{a}; \mathbf{b} | r)$ is the Meijer G-function. Essentially our assumptions only require that $(-\Delta)^{\alpha/2} f(x)$ is well defined as a singular integral, and we allow for any $\alpha > -d$.}

\mk{By specifying some or all of the parameters, we obtain expressions for $(-\Delta)^{\alpha/2} f(x)$ for a wide collection of functions $f(x) = V(x) \phi(|x|^2)$. For example, we can let $\phi(r) = r^\rho (1 + r)^\sigma$, $\phi(r) = r^\rho (1 - r)_+^\sigma$ or $\phi(r) = r^\rho (r - 1)_+^\sigma$ (Corollaries~\ref{cor:power} and~\ref{cor:powerball}). We also obtain an elegant expression for the generalised hypergeometric function: if $\mathbf{b}' = (b_1, \dots, b_{q-1})$ and
\begin{equation}\label{eq:VF}
f(x):=V(x) {_pF_q}\Big(\,\begin{matrix}
\multicolumn{2}{c}{\mathbf{a}} \\
\mathbf{b}', & \tfrac{d + 2l}{2}
\end{matrix}\, \Big \vert \, {-|x|^2} \Big) ,
\end{equation}
then
\begin{equation}\label{eq:DVF}
(-\Delta)^{\alpha/2} f(x)=
\frac{2^{\alpha} \prod\limits_{j = 1}^p \Gamma(a_j + \tfrac{\alpha}{2}) \prod\limits_{j = 1}^{q-1} \Gamma(b_j)}{\prod\limits_{j = 1}^p \Gamma(a_j) \prod\limits_{j = 1}^{q-1} \Gamma(b_j + \tfrac{\alpha}{2})} \, V(x) {_pF_q}\Big(\,\begin{matrix}
\multicolumn{2}{c}{\mathbf{a} + \tfrac{\alpha}{2}} \\
\mathbf{b}' + \tfrac{\alpha}{2}, & \tfrac{d + 2l}{2}
\end{matrix}\, \Big \vert \, {-|x|^2} \Big) 
\end{equation}
(Corollary~\ref{cor:DVF}). Finally, if $P^{(\alpha,\beta)}_n(r)$ is the Jacobi polynomial and
\begin{equation}\label{eq:VP}
 f(x) := (1 - |x|^2)_+^{\alpha/2} V(x) P^{(\alpha/2,d/2+l-1)}_n(2 |x|^2 - 1) ,
\end{equation}
then
\begin{equation}\label{eq:DVP}
 (-\Delta)^{\alpha/2} f(x) = \frac{2^\alpha \Gamma(1 + \tfrac{\alpha}{2} + n) \Gamma(\tfrac{d + 2 l + \alpha}{2} + n)}
{n! \, \Gamma(\tfrac{d + 2l}{2} + n)} \, V(x) P^{(\alpha/2,d/2+l-1)}_n(2 |x|^2 - 1)
\end{equation}
in the unit ball, that is, $(1 - |x|^2)_+^{\alpha/2} (-\Delta)^{\alpha/2} f(x)$ is a multiple of $f(x)$ (Theorem~\ref{thm:DVP}).}

\subsection{Riesz potential operator and the fractional Laplace operator}\label{subsection_Riesz_operator}

Everywhere in this paper we assume that $d \ge 1$ is the dimension. For $\alpha \in (0,d)$, the \emph{Riesz potential operator} is defined by
\begin{equation}\label{def_I_alpha}
 (-\Delta)^{-\alpha/2} f(x) := \frac{1}{\gamma_d(\alpha)}\int_{\r^d} \frac{f(x-y)}{|y|^{d - \alpha}}  \d y,
\end{equation}
where 
\begin{equation}
\gamma_{d}(\alpha):=\frac{2^\alpha \pi^{d/2} \Gamma(\tfrac{\alpha}{2})}{\Gamma(\tfrac{d - \alpha}{2})}.
\end{equation}
When $0<\alpha<2$ we define the {\it fractional Laplace operator} (also known as the \emph{fractional Riesz derivative}) as
\begin{equation}\label{def_fractional_Laplacian_0_alpha_2}
(-\Delta)^{\alpha/2}f(x):= \frac{1}{|\gamma_d(-\alpha)|}
\lim_{\varepsilon \to 0^+} \int_{\r^d \setminus B(0, \varepsilon)} 
 \frac{f(x)-f(x-y)}{|y|^{d+\alpha}} \d y,
\end{equation}
where $B(0,\varepsilon)$ denotes the ball of radius $\varepsilon$ centered at the origin. 
This definition can be extended to $\alpha \ge 2$, though it requires more work. Following~\cite{Samko_book}, we define the centered difference 
\[
\Delta^k_{y} f(x):=\sum\limits_{j=0}^k (-1)^{j} \tbinom{k}{j} f(x+(\tfrac{k}{2}-j)y)
\]
and introduce
\[
\chi_{d,k}(\alpha) := -\gamma_d(-\alpha) \sum\limits_{j=0}^k (-1)^j \tbinom{k}{j} |\tfrac{k}{2}-j|^{\alpha}.
\]
Choose $k$ to be an even integer. It is known that 
$\chi_{d,k}(\alpha)$ is analytic and zero-free in the region $0<\re(\alpha)<k$ (the poles of \mk{$\gamma_d(-\alpha)$} when $\alpha$ is an even integer are canceled by the corresponding zeros \mk{of the sum}, see Theorem~3.4 in~\cite{Samko_book}). We define the fractional Laplace operator as 
\begin{equation}\label{def_fractional_Laplacian}
(-\Delta)^{\alpha/2}f(x) := \frac{1}{\chi_{d,k}(\alpha)}
\lim_{\varepsilon \to 0^+} \int_{\r^d \setminus B(0, \varepsilon)} 
 \frac{-\Delta^k_y f(x)}{|y|^{d+\alpha}} \d y, 
\end{equation}
where $k$ is an even integer strictly greater than $\re(\alpha)$. When $0<\alpha<2$, this definition is easily seen to be equivalent to 
\eqref{def_fractional_Laplacian_0_alpha_2} (provided that $(1 + |y|)^{-d - \alpha} f(y)$ is in integrable). For a different representation of $(-\Delta)^{\alpha/2} f(x)$, based on the Pizzetti's formula, see formula~(1.1.10) in~\cite{bib:l72}.

For Schwartz functions $f$ the Riesz potential and the fractional Laplacian can be defined in an alternative way 
\begin{equation}\label{def_I_alpha_Fourier}
 (-\Delta)^{\alpha/2} f = \fourier_d^{-1} (\hat k_{\alpha} \fourier_d f),
\end{equation}
where $\hat k_{\alpha}(x)=|x|^{\alpha}$ and $\fourier_d$ denotes the Fourier transform in $\r^d$,
\begin{equation}\label{def_Fourier}
\fourier_d f(x):=\int_{\r^d} e^{\i x\cdot y} f(y) \d y.
\end{equation}
Formula \eqref{def_I_alpha_Fourier} shows that when $0<\re(\alpha)<d$, the fractional Laplace operator is the inverse of the Riesz potential operator for nice enough functions, \mk{which explains the notation $(-\Delta)^{-\alpha/2}$}. This is made precise in Theorem~3.24 in~\cite{Samko_book} in the setting of $L^p$ spaces, see also~\cite{bib:k:lap,Stein_book}. We will need the following pointwise version of this result.

\begin{proposition}[Proposition~7.2 in~\cite{bib:k:lap}]\label{prop_Delta_inverse_I_alpha}
Assume that $0 < \alpha < d$, $(1 + |y|)^{\alpha - d} g(y)$ is integrable and $f(x)=(-\Delta)^{-\alpha/2} g(x)$. If $g$ is continuous at some $x$, then $g(x) = (-\Delta)^{\alpha/2} f(x)$.
\end{proposition}



\subsection{Hypergeometric function}

\mk{Let $p$, $q$ be nonnegative integers such that $p \le q + 1$, and let $\mathbf{a} = (a_1, \dots, a_p) \in \c^p$ and $\mathbf{b} = (b_1, \dots, b_q) \in \c^q$. The generalised hypergeometric function is defined as the hypergeometric series
\[
{_pF_q}\Big(\,\begin{matrix}
\mathbf{a} \\
\mathbf{b}
\end{matrix}\, \Big \vert \, r \Big) := \sum_{n = 0}^\infty
\frac{\prod\limits_{j = 1}^p a_j^{(n)}}{\prod\limits_{j = 1}^q b_j^{(n)}} 
\, \frac{r^n}{n!} \, ,
\]
as long as none of the parameters $b_j$ is a non-positive integer; here $c^{(n)} = c (c + 1) \dots (c + n - 1)$ denotes the rising factorial. If $p \le q$, then the above series is convergent for all $r \in \c$ and ${_pF_q}$ is an entire function. In the important case $p = q + 1$, the series converges when $|r| < 1$, but ${_pF_q}$ extends to an analytic function in $\c \setminus [1, \infty)$.}

The regularised hypergeometric function $_p\tilde{F}_q$ is defined as
\begin{equation*}
{_p\tilde F_q}\Big(\,\begin{matrix}
\mathbf{a} \\
\mathbf{b}
\end{matrix}\, \Big \vert \, r \Big) :=
\frac{1}{\prod\limits_{j = 1}^q \Gamma(b_j)} \, {_pF_q}\Big(\,\begin{matrix}
\mathbf{a} \\
\mathbf{b}
\end{matrix}\, \Big \vert \, r \Big) .
\end{equation*}
Here none of $b_j$ is a non-positive integer, but ${_p\tilde{F}_q}$ extends analytically to arbitrary values of~$b_j$. For more information on hypergeometric functions, see~\cite{bib:gr07,Prudnikov_V3}.

\subsection{Meijer G-function}\label{subsection_Meijer}

In this section we define the Meijer G-function \mk{and discuss some of its properties}. We begin with four non-negative integers $m$, $n$, $p$ and $q$ and two vectors $\mathbf{a}=(a_1,\dots,a_p) \in \c^p$ and $\mathbf{b}=(b_1,\dots,b_q) \in \c^q$, and define  
\begin{equation}\label{eq:MG}
\cG^{mn}_{pq}\Big( 
\,\begin{matrix}
\mathbf{a} \\ \mathbf{b}
\end{matrix}\, \Big  \vert \, s 
\Big ) : =  \frac{\prod\limits_{j=1}^m \Gamma(b_j+s) \prod\limits_{j=1}^n \Gamma(1-a_j-s)}
{\prod\limits_{j=m+1}^q \Gamma(1-b_j-s) \prod\limits_{j=n+1}^p \Gamma(a_j+s)}. 
\end{equation}
This will serve as the Mellin transform of the Meijer G-function. We denote 
\begin{equation}\label{eq:nu}
 \nu = \sum_{j=1}^p \re(a_j)-\sum_{j=1}^q \re(b_j)
\end{equation}
and
\begin{equation}\label{def_bar_a}
\begin{aligned}
\underline{b}:=\min\limits_{1\le j \le m} \re(b_j), &\qquad& \overline{a}:=\max\limits_{1\le j \le n} \re(a_j) ;
\end{aligned}
\end{equation}
we set $\underline{b}=+\infty$ \mk{if $m = 0$} and $\overline{a}=-\infty$ \mk{if $n = 0$}. 

\mk{Throughout the entire article, when speaking about the Meijer G-function,} we will \mk{tacitly} assume that $b_i - a_j$ is not a positive integer for $i = 1, \dots, m$ and $j = 1, \dots, n$, so that no pole of $\Gamma(b_i+s)$ coincides with a pole of $\Gamma(1 - a_j - s)$; \mk{otherwise, the Meijer G-function is not defined. We will also always assume that $p + q \le 2m + 2n$, so that there are at least as many gamma functions in the numerator of~\eqref{eq:MG} as there are in the denominator.} We also introduce the following five conditions on parameters $m$, $n$, $p$, $q$, $\mathbf{a}$ and $\mathbf{b}$ that will be required for some statements:
\begin{equation*}
\begin{aligned}
&\text{Condition~S:} && 1 - \overline{a} > -\underline{b}, \\
&\text{Condition A:} && p+q<2m+2n, \\
&\text{Condition B:} && p+q=2m+2n , \; p = q \\
&\text{Condition C:} && p+q=2m+2n , \; p < q \\
&\text{Condition D:} && p+q=2m+2n , \; p > q .
\end{aligned}
\end{equation*}
Note that Conditions~A through~D are mutually exclusive. \mk{Before we define the Meijer G-function, we discuss the role of the above conditions}.

Condition~S is needed because it separates the poles of $\Gamma(b_j+s)$ from the poles of $\Gamma(1-a_j-s)$ in the numerator in \eqref{eq:MG}, thus the function $\cG^{mn}_{pq}( \mathbf{a} ; \mathbf{b} | s)$ is analytic in $s$ in the strip $-\underline b<\re(s)<1-\overline a$. By Stirling's asymptotic formula for the gamma function, for $\lambda \in \r$,
\[
 \lim_{t \to \pm \infty} \frac{e^{|t| \pi / 2} |\Gamma(\lambda + \i t)|}{|t|^{\lambda - 1/2}} = \sqrt{2 \pi}
\]
(see formula~8.328.1 in~\cite{bib:gr07}). Therefore, Condition~A ensures that $\cG^{mn}_{pq}( \mathbf{a} ; \mathbf{b} | s)$ converges to zero exponentially fast as $|\im s| \to \infty$ within the strip $-\underline b<\re(s)<1-\overline a$. When Condition~B, C or~D is satisfied, then the exponential parts of the gamma functions cancel out, and one can check that for every $\varepsilon > 0$ the function $|\cG^{mn}_{pq}( \mathbf{a} ; \mathbf{b} | s)|$ is of smaller order than $|\im s|^{-\nu - (p - q) (\lambda - 1/2) + \varepsilon}$ as $|\im s| \to \infty$ along the line $\lambda + \i \r$. In particular, this function is integrable if $\nu - (p - q) (\lambda - \tfrac{1}{2}) > 1$.

\begin{definition}
Assume that parameters $m$, $n$, $p$, $q$, $\mathbf{a}$ and $\mathbf{b}$ satisfy Condition~S and either of Conditions~A through~D. \mk{Suppose in addition that: $\nu > 1$ if Condition~B is satisfied; $\nu > 1 + (q - p) (-\underline{b} - \tfrac{1}{2})$ if Condition~C holds; and $\nu > 1 - (p - q) (\tfrac{1}{2} - \overline{a})$ if Condition~D holds}. We define the \emph{Meijer G-function} as \mk{the inverse Mellin transform}
\begin{equation}\label{def_Meijer_G}
G^{mn}_{pq}\Big( 
\,\begin{matrix}
\mathbf{a} \\ \mathbf{b}
\end{matrix}\, \Big \vert \, r
\Big ):=\frac{1}{2\pi \i} \int_{\lambda + \i \r} 
\cG^{mn}_{pq}\Big( 
\,\begin{matrix}
\mathbf{a} \\ \mathbf{b}
\end{matrix}\, \Big  \vert \, s 
\Big ) r^{-s} \d s,  
\end{equation}
where $r>0$ and $\lambda \in (-\underline{b}, 1-\overline{a})$; \mk{if Condition~C or~D is satisfied, we also require that $\nu + (p - q) (\lambda - \tfrac{1}{2}) > 1$}. 
\end{definition}

\mk{For any $\lambda$ as in the above definition, the Meijer G-function is defined as the inverse Mellin transform of an absolutely integrable function along the line $\lambda + \i \r$. In particular, the Meijer G-function is then bounded by $C r^{-\lambda}$ for some $C > 0$. Therefore, if we denote
\begin{equation}\label{eq:indices}
\begin{aligned}
 \underline{\lambda} &= \begin{cases} -\underline{b} & \text{under Conditions~A, B or C} \\ \max(-\underline{b}, \tfrac{1}{2} - \tfrac{\nu - 1}{p - q}) & \text{under Condition~D}, \end{cases} \\
 \overline{\lambda} &= \begin{cases} 1 - \overline{a} & \text{under Conditions~A, B or D} \\ \min(1 - \overline{a}, \tfrac{1}{2} + \tfrac{\nu - 1}{q - p}) & \text{under Condition~C}, \end{cases}
\end{aligned}
\end{equation}
then we immediately have that for any $\varepsilon>0$,
\begin{equation}\label{eq:G:asymp}
 G^{mn}_{pq}\Big( 
\,\begin{matrix}
\mathbf{a} \\ \mathbf{b}
\end{matrix}\, \Big  \vert \, r 
\Big )= 
\begin{cases}
O(r^{-\underline{\lambda}-\varepsilon}) & \text{as $r\to 0^+$,}\\
O(r^{-\overline{\lambda}+\varepsilon}) & \text{as $r\to +\infty$.}
\end{cases}
\end{equation}
Intuitively, the parameter $\nu$ describes the regularity of the Meijer G-function near $1$ when Condition~B is satisfied, its oscillations at $\infty$ when Condition~C holds, and its oscillations near $0$ when Condition~D holds. Note that in each case the corresponding regularity improves as $\nu$ increases.}

\mk{The following transformation rules also follow easily from the definition as the inverse Mellin transform:
\begin{align}
\label{eq:G:shift}
r^c \, G^{mn}_{pq}\Big( 
\,\begin{matrix}
\mathbf{a} \\ \mathbf{b}
\end{matrix}\, \Big  \vert \, r 
\Big ) & = G^{mn}_{pq}\Big( 
\,\begin{matrix}
\mathbf{a}+c \\ \mathbf{b}+c
\end{matrix}\, \Big  \vert \, r 
\Big ), \\
\label{eq:G:inverse}
G^{mn}_{pq}\Big( 
\,\begin{matrix}
\mathbf{a} \\ \mathbf{b}
\end{matrix}\, \Big  \vert \, r^{-1}
\Big ) & = G^{nm}_{qp}\Big( 
\,\begin{matrix}
1-\mathbf{b} \\
1- \mathbf{a}
\end{matrix}\, \Big  \vert \, r
\Big ), \\
\label{reduction1}
G^{mn}_{pq}\Big(
\,\begin{matrix}
c, & \mathbf{a}' \\ \mathbf{b}', & c
\end{matrix}\, \Big  \vert \, r
\Big ) & = G^{m,n-1}_{p-1,q-1}\Big( 
\,\begin{matrix}
\mathbf{a}' \\ \mathbf{b}'
\end{matrix}\, \Big  \vert \, r 
\Big ), \\
\label{reduction2}
G^{mn}_{pq}\Big(
\,\begin{matrix}
\mathbf{a}', & c \\ c, & \mathbf{b}'
\end{matrix}\, \Big  \vert \, r
\Big ) & = G^{m-1,n}_{p-1,q-1}\Big( 
\,\begin{matrix}
\mathbf{a}' \\ \mathbf{b}'
\end{matrix}\, \Big  \vert \, r 
\Big ).
\end{align}
where $\mathbf{a}' = (a_1, \dots, a_{p-1})$ and $\mathbf{b}' = (b_1, \dots, b_{q-1})$.}

The above definition of Meijer G-function is not the most general possible. One \mk{can} relax Condition~S and restrictions on $\nu$ by appropriately deforming the contour of integration in~\eqref{def_Meijer_G}, see Chapter~8.2 in~\cite{Prudnikov_V3} for more details. Another way of extending the definition of Meijer G-function is provided by expansion in terms of generalised hypergeometric functions, \mk{described briefly below}.

\mk{Assume that $b_j-b_k$ is not an integer for $1\le j<k \le m$. If $p<q$ or $p=q$ and $|r|<1$, then we have
\begin{equation}\label{eq:G:F}
\begin{aligned}
G^{mn}_{pq}\Big( 
\,\begin{matrix}
\mathbf{a}  \\ \mathbf{b} 
\end{matrix}\, \Big  \vert \, r 
\Big )& =
\sum\limits_{k=1}^m 
\frac{\pi^{m-1} \prod\limits_{j=1}^n \Gamma(1+b_k-a_j)}
{\smash{\prod\limits_{\substack{1\le j \le m \\ j\ne k}} \sin((b_j-b_k) \pi)} \prod\limits_{j=n+1}^p \Gamma(a_j-b_k)} \times \\
& \hspace*{13em} \times 
r^{b_k}{_p\tilde{F}_{q-1}}
\Big( 
\,\begin{matrix}
1 + b_k - \mathbf{a} \\ 1 + b_k - \mathbf{b}'_k 
\end{matrix}\, \Big  \vert \, (-1)^{p-m-n} r 
\Big ), 
\end{aligned}
\end{equation}
where $\mathbf{b}'_k = (b_1, \dots, b_{k-1}, b_{k+1}, \dots, b_q)$, see~\cite{bib:gr07,bib:ms73}. If $p>q$ or $p=q$ and $|r|>1$, the corresponding representation of Meijer G-function in terms of  ${}_q\tilde{F}_{p-1}$ functions can be obtained using~\eqref{eq:G:inverse} and~\eqref{eq:G:F}. Furthermore, if $b_j = b_k$ for some $j \ne k$, a similar, but much more complicated expansion holds true, see Chapter~V in~\cite{bib:ms73}.}

\mk{The Meijer G-function, whenever defined, is analytic in all parameters $a_j$ and $b_j$, and it has the properties~\eqref{eq:G:asymp} through~\eqref{reduction2}}. Under Condition~A, formula~\eqref{def_Meijer_G} in fact defines the Meijer G-function as an analytic function in a sector $|\arg r|<(m+n-(p+q)/2) \pi$. When condition~B holds true (with arbitrary $\nu$), the Meijer G-function extends to an analytic function in the unit disk $|r| < 1$ and in the complement of the unit disk $|r| > 1$. \mk{Under Conditions~C and~D (or more generally, when $p \ne q$), the Meijer G-function is analytic in $|\arg r| < \infty$, the Riemann surface of the logarithm. For a comprehensive description of the Meijer G-function, we refer to~\cite{bib:ms73} and~\cite{Prudnikov_V3}.}

\subsection{Main results}\label{subsection_Main_results}

We define \emph{solid harmonic polynomials} of order $l$ to be polynomials $V(x)$ in $\r^d$ homogeneous of degree $l$, which satisfy $\Delta V=0$. In dimension $d=1$ there are only two solid harmonic polynomials, $V(x)\equiv 1$ and $V(x)\equiv x$. \mk{When $d = 2$, solid harmonic polynomials can be obtained as real or imaginary parts of monomials (if $\r^2$ is identified with $\c$).} 
The main idea behind our results is stated in the following proposition\mk{, which is a pointwise version of an $L^2$ result given as Lemma~24.8 in~\cite{bib:r96}}. 

\begin{proposition}
\label{prop_riesz_mellin}
Let $V(x)$ be a solid harmonic polynomial of degree $l\ge 0$ and $\delta = d + 2 l$. Suppose that
\[
f(x) = V(x) \phi(|x|^2) ,
\]
where $\phi$ is given as the (absolutely convergent) inverse Mellin transform
\begin{equation}
\label{eqn_inverse_mellin}
\phi(r) := \frac{1}{2 \pi \i} \int_{\lambda + \i \r} \mellin \phi(s) r^{-s} \d s
\end{equation}
for some $\lambda \in \r$. If
\begin{equation}
\label{eqn_riesz_mellin_condition}
0 < \alpha < 2 \lambda - l < d ,
\end{equation}
then \mk{the Riesz potential} $(-\Delta)^{-\alpha/2} f(x)$ is well-defined for $x \ne 0$, and
\[
(-\Delta)^{-\alpha/2} f(x) = V(x) \psi(|x|^2),
\]
where
\begin{equation}
\label{eqn_riesz_mellin}
\psi(r) := \frac{1}{2 \pi \i} \int_{\lambda - \alpha/2 + \i \r} \frac{\Gamma(s) \Gamma(\tfrac{\delta - \alpha}{2} - s)}{2^\alpha \Gamma(\tfrac{\alpha}{2} + s) \Gamma(\tfrac{\delta}{2} - s)} \, \mellin \phi(s + \tfrac{\alpha}{2}) r^{-s} \d s .
\end{equation}
\end{proposition}

\noindent
In other words, the Mellin transform of $\psi$ satisfies
\[
\mellin \psi(s) = \frac{\Gamma(s) \Gamma(\tfrac{\delta - \alpha}{2} - s)}{2^\alpha \Gamma(\tfrac{\alpha}{2} + s) \Gamma(\tfrac{\delta}{2} - s)} \, \mellin \phi(s + \tfrac{\alpha}{2}) \, .
\]
The above formula fits perfectly into the world of Meijer G-function: if $\phi$ is a Meijer G-function, then so is $\psi$.
We recall that $\overline{a}$ and $\underline{b}$ are defined in 
\eqref{def_bar_a}, \mk{while $\overline{\lambda}$ and $\underline{\lambda}$ are defined in~\eqref{eq:indices}}. The following two theorems are our main results about Meijer G-function. \mk{For technical reasons, we consider the Riesz potential and the fractional Laplace operators separately.}

\begin{theorem}\label{thm:IVG}
Let $V(x)$ be a solid harmonic polynomial of degree $l\ge 0$ and $\delta = d + 2 l$. 
Assume that $0<\alpha<d$ and
parameters $m$, $n$, $p$, $q$, $\mathbf{a}$ and $\mathbf{b}$ satisfy \mk{Condition~A}, as well as
\begin{equation}\label{eq:IVG:cond}
\begin{aligned}
2 (1 - \overline{a}) & > \alpha + l , &\qquad -2 \underline{b} & < d + l .
\end{aligned}
\end{equation}
Define
$
f(x):=V(x) G^{mn}_{pq}\Big( 
\,\begin{matrix}
\mathbf{a} \\ \mathbf{b}
\end{matrix}\, \Big  \vert \, |x|^2 
\Big).$
Then
\begin{equation}\label{eq:DVG:I}
(-\Delta)^{-\alpha/2} f(x)=2^{-\alpha} V(x) G^{m+1,n+1}_{p+2,q+2}\Big( 
\,\begin{matrix}
1-\tfrac{\delta - \alpha}{2}, &\mathbf{a}+\tfrac{\alpha}{2}, & \tfrac{\alpha}{2} \\ 0, & \mathbf{b}+\tfrac{\alpha}{2}, &1-\tfrac{\delta}{2}
\end{matrix}\, \Big  \vert \, |x|^2 
\Big )
\end{equation}
for all $x \ne 0$. \mk{The same statement holds under Conditions~B, C and~D, provided that
\begin{equation}\label{eq:IVG:cond2}
\begin{aligned}
2 \overline{\lambda} & > \alpha + l , &\qquad 2 \underline{\lambda} & < d + l ;
\end{aligned}
\end{equation}
if Condition~B is satisfied, we additionally require that $\nu > 0$ and either $|x| \ne 1$ or $\nu + \alpha > 1$.}
\end{theorem}

\begin{theorem}\label{thm:DVG}
Let $V(x)$ be a solid harmonic polynomial of degree $l\ge 0$ and $\delta = d + 2 l$. Assume that $\alpha>0$ and parameters $m$, $n$, $p$, $q$, $\mathbf{a}$ and $\mathbf{b}$ satisfy Condition~A, and in addition
\begin{equation}\label{eq:DVG:cond}
\begin{aligned}
2 (1 - \overline{a}) & > -\alpha + l , & -2 \underline{b} & < d + l .
\end{aligned}
\end{equation}
Define
$
f(x):=V(x) G^{mn}_{pq}\Big( 
\,\begin{matrix}
\mathbf{a} \\ \mathbf{b}
\end{matrix}\, \Big  \vert \, |x|^2 
\Big).
$
Then
\begin{equation}\label{eq:DVG:D}
(-\Delta)^{\alpha/2} f(x)=2^{\alpha} V(x) G^{m+1,n+1}_{p+2,q+2}\Big( 
\,\begin{matrix}
1-\tfrac{\delta + \alpha}{2}, & \mathbf{a}-\tfrac{\alpha}{2}, & -\tfrac{\alpha}{2} \\ 0, & \mathbf{b}-\tfrac{\alpha}{2}, &1-\tfrac{\delta}{2}
\end{matrix}\, \Big  \vert \, |x|^2 
\Big )
\end{equation}
for all $x \ne 0$. \mk{The same statement holds under Conditions~B, C and~D, provided that
\begin{equation}\label{eq:DVG:cond2}
\begin{aligned}
2 \overline{\lambda} & > -\alpha + l , &\qquad 2 \underline{\lambda} & < d + l ;
\end{aligned}
\end{equation}
if Condition~B is satisfied, we additionally require that $\nu > 0$ and either $|x| \ne 1$ or $\nu > 1 + \alpha$. Finally, the result extends to $x = 0$ whenever both $f$ and $(-\Delta)^{\alpha/2} f$ are continuous at $0$.}
\end{theorem}

By considering special cases of Meijer G-functions and using Theorems~\ref{thm:IVG} and~\ref{thm:DVG} one \mk{can} obtain a multitude of new explicit examples and give a short proof of many known results. The most useful resource for deriving this results is the large collection of formulas, expressing known elementary and special functions in terms of Meijer G-function, which can be found in Table~8.4 in~\cite{Prudnikov_V3}. In the next two sections we provide several examples with different level of specialisation. Whenever possible, we state our results in terms of the regularised hypergeometric function ${_p\tilde{F}_q}$, a somewhat simpler object than the Meijer G-function. \mk{Note that} if none of the $a_j$ is a non-positive integer, we have, by~\eqref{eq:G:F},
\begin{equation}
\label{eq:FG}
{_p\tilde F_q}\Big(\,\begin{matrix}
\mathbf{a} \\
\mathbf{b}
\end{matrix}\, \Big \vert \, {-r} \Big) =
\frac{1}{\prod\limits_{j = 1}^p \Gamma(a_j)} \, G^{1,p}_{p,q+1}\Big( 
\,\begin{matrix}
\multicolumn{2}{c}{1-\mathbf{a}} \\ 0,& 1-\mathbf{b}
\end{matrix}\, \Big  \vert \, r
\Big ).
\end{equation}

\subsection{Full space}

By formula~8.4.2.5 in~\cite{Prudnikov_V3} and the property~\eqref{eq:G:shift}, we have
\begin{equation}
\label{eq:power:G}
r^\rho (1+r)^\sigma = \frac{1}{\Gamma(-\sigma)} \,
G^{11}_{11}\Big( 
\,\begin{matrix}
1 + \rho + \sigma \\ \rho
\end{matrix}\, \Big  \vert \, r
\Big).
\end{equation}
This implies the following extension of expressions given in~\cite{Samko_book}. We remark that when $\rho = 0$ or $2 \rho + 2 \sigma = \alpha - \delta$, then the expression for $(-\Delta)^{\alpha/2} f(x)$ can be written in terms of the Gauss's hypergeometric function ${_2F_1}$.

\begin{corollary}
\label{cor:power}
Let $V(x)$ be a solid harmonic polynomial of degree $l\ge 0$ and $\delta = d + 2 l$. Assume that $-d < \alpha < 0$ or $\alpha > 0$, $2 \rho > -d - l$ and $2 \rho + 2 \sigma < \alpha - l$. Define
\[
f(x) := V(x) |x|^{2 \rho} (1+|x|^2)^\sigma .
\]
Then
\[
(-\Delta)^{\alpha/2} f(x) = \frac{2^{\alpha}}{\Gamma(-\sigma)} \, V(x) G^{2,2}_{3,3}\Big( 
\,\begin{matrix}
1-\tfrac{\delta + \alpha}{2}, &1 + \rho + \sigma - \tfrac{\alpha}{2}, & -\tfrac{\alpha}{2} \\ 0, &\rho - \tfrac{\alpha}{2}, &1-\tfrac{\delta}{2}
\end{matrix}\, \Big  \vert \, |x|^2 
\Big )
\]
for all $x \ne 0$, and also for $x = 0$ if $2 \rho > \alpha - l$ or $2 \rho$ is an even integer not less than $-l$.
\end{corollary}

Due to~\eqref{eq:FG} (and the reduction formulas~\eqref{reduction1} and~\eqref{reduction2}), the operator $(-\Delta)^{\alpha/2}$ acts nicely on generalised hypergeometric functions. In the following statement the last parameter $b_q = \tfrac{\delta}{2}$ has a special role; to avoid ambiguities, we denote $\mathbf{b}' = (b_1, \dots, b_{q-1})$.

\begin{corollary}\label{cor:DVF}
Let $V(x)$ be a solid harmonic polynomial of degree $l\ge 0$ and $\delta = d + 2 l$. Assume that $\alpha > -d$ and parameters $p$, $q$, $\mathbf{a} \in \c^p$ and $\mathbf{b}' \in \c^{q-1}$ satisfy
\[
\begin{aligned}
p \in \{ q - 1 , q, q + 1 \} ,
&\qquad& 2 \underline{a} > -\alpha + l ,
\end{aligned}
\]
where $\underline{a} = \min\limits_{1 \le j \le p} \re(a_j)$. Define 
$
f(x):=V(x) {_p\tilde{F}_q}\Big(\,\begin{matrix}
\multicolumn{2}{c}{\mathbf{a}} \\
\mathbf{b}', & \tfrac{\delta}{2}
\end{matrix}\, \Big \vert \, {-|x|^2} \Big)
$.
Then
\[
(-\Delta)^{\alpha/2} f(x)=
\frac{2^{\alpha} \prod\limits_{j = 1}^p \Gamma(a_j + \tfrac{\alpha}{2})}{\prod\limits_{j = 1}^p \Gamma(a_j)} \, V(x) {_p\tilde{F}_q}\Big(\,\begin{matrix}
\multicolumn{2}{c}{\mathbf{a} + \tfrac{\alpha}{2}} \\
\mathbf{b}' + \tfrac{\alpha}{2}, & \tfrac{\delta}{2}
\end{matrix}\, \Big \vert \, {-|x|^2} \Big)
\]
for all $x\in \r^d$.
\end{corollary}

The requirement that $b_q = \tfrac{\delta}{2}$ is of purely notational nature. In order to use Corollary~\ref{cor:DVF} in the general case, simply write
\[
{_p\tilde{F}_q}\Big(\,\begin{matrix}
\mathbf{a} \\
\mathbf{b}
\end{matrix}\, \Big \vert \, r \Big) = \Gamma(\tfrac{\delta}{2}) \, {_{p+1}\tilde{F}_{q+1}}\Big(\,\begin{matrix}
\mathbf{a}, & \tfrac{\delta}{2} \\[0.15em]
\mathbf{b}, & \tfrac{\delta}{2}
\end{matrix}\, \Big \vert \, r \Big).
\]
\mk{We remark that ${_0\tilde{F}_1}(a \, | \, {-|x|^2}) = |x|^{2 - 2a} J_{a-1}(2 |x|)$, where $J_{a-1}$ is the Bessel function. In particular, ${_0\tilde{F}_1}(\tfrac{1}{2} \, | \, {-|x|^2}) = \pi^{-1/2} \cos(2 |x|)$, which gives a very surprising expression: if $f(x) := \cos(|x|)$, then
\[
 (-\Delta)^{\alpha/2} f(x) = \frac{\sqrt{\pi} \, \Gamma(\tfrac{d + \alpha}{2})}{\Gamma(\tfrac{1 + \alpha}{2}) \Gamma(\tfrac{d}{2})} \, {_1\tilde{F}_2}
\Big(\,\begin{matrix}
\multicolumn{2}{c}{\tfrac{d + \alpha}{2}} \\
\tfrac{1 + \alpha}{2} , & \tfrac{d}{2}
\end{matrix}\, \Big \vert \, {-\tfrac{1}{2} |x|^2} \Big)
\]
for all $\alpha > 0$ and all $x \in \r^d$.}

\subsection{Unit ball and its complement}

The function $f$ in Theorems~\ref{thm:IVG} and~\ref{thm:DVG} is supported in the unit ball when $p = q = m$ and $n = 0$, and in the complement of the unit ball when $p = q = n$ and $m = 0$, see~\eqref{eq:G:F}. Therefore, we obtain explicit expressions for $(-\Delta)^{\alpha/2} f_1(x)$ and $(-\Delta)^{\alpha/2} f_2(x)$ (with $|x| \ne 1$), where
\begin{align*}
f_1(x) & := V(x) G^{p0}_{pp}\Big( 
\,\begin{matrix}
\mathbf{a} \\ \mathbf{b}
\end{matrix}\, \Big  \vert \, |x|^2 
\Big) , &
f_2(x):=V(x) G^{0p}_{pp}\Big( 
\,\begin{matrix}
\mathbf{a} \\ \mathbf{b}
\end{matrix}\, \Big  \vert \, |x|^2 
\Big)
\end{align*}
are supported in the unit ball and its complement, respectively. Here we assume that Condition~B holds, $\nu > 0$ and $-2 \underline{b} < d + l$ in the expression for $(-\Delta)^{\alpha/2} f_1(x)$ with $x \ne 0$ (a more restrictive condition is needed when $x = 0$), and that Condition~B holds, $\nu > 0$ and $2 (1 - \overline{a}) > -\alpha + l$ in the expression for $(-\Delta)^{\alpha/2} f_2(x)$.

For example, formulas 8.4.2.3--4 in \cite{Prudnikov_V3} combined with the property~\eqref{eq:G:shift} tell us that
\begin{align}
\label{eq:powerball:G}
r^\rho (1-r)_+^\sigma = \Gamma(1 + \sigma)
G^{10}_{11}\Big( 
\,\begin{matrix}
1 + \rho + \sigma \\ \rho
\end{matrix}\, \Big  \vert \, r
\Big), \\
\label{eq:powerballc:G}
r^\rho (r-1)_+^\sigma = \Gamma(1 + \sigma)
G^{01}_{11}\Big( 
\,\begin{matrix}
1 + \rho + \sigma \\ \rho
\end{matrix}\, \Big  \vert \, r
\Big).
\end{align}
Theorems~\ref{thm:IVG} and~\ref{thm:DVG} readily imply the following explicit expressions, which extend those previously derived in~\cite{bib:bik11,bib:d12} (see also~\cite{bib:huang2014}). 

\begin{corollary}
\label{cor:powerball}
Let $V(x)$ be a solid harmonic polynomial of degree $l\ge 0$ and $\delta = d + 2 l$. Assume that $-d < \alpha < 0$ or $\alpha > 0$. Define
\[
\begin{aligned}
f_1(x) & := V(x) |x|^{2 \rho} (1-|x|^2)_+^\sigma , \\
f_2(x) & := V(x) |x|^{2 \rho} (|x|^2-1)_+^\sigma .
\end{aligned}
\]
\begin{enumerate}[label={\rm{(\roman*)}}]
\item
If $2 \rho > -d - l$ and $\sigma > -1$, then
\[
(-\Delta)^{\alpha/2} f_1(x) = 2^{\alpha} \Gamma(1 + \sigma) V(x) G^{2,1}_{3,3}\Big( 
\,\begin{matrix}
1-\tfrac{\delta + \alpha}{2}, &1 + \rho + \sigma - \tfrac{\alpha}{2}, & -\tfrac{\alpha}{2} \\ 0, &\rho-\tfrac{\alpha}{2}, &1-\tfrac{\delta}{2}
\end{matrix}\, \Big  \vert \, |x|^2 
\Big )
\]
for all $x$ such that $x \ne 0$ and $|x| \ne 1$, and also for $x = 0$ if $2 \rho > \alpha - l$ or $2 \rho$ is an even integer not less than $-l$.
\item
If $2 \rho + 2 \sigma < \alpha - l$ and $\sigma > -1$, then
\[
(-\Delta)^{\alpha/2} f_2(x) = 2^{\alpha} \Gamma(1 + \sigma) V(x) G^{1,2}_{3,3}\Big( 
\,\begin{matrix}
1-\tfrac{\delta + \alpha}{2}, &1 + \rho + \sigma -\tfrac{\alpha}{2}, & -\tfrac{\alpha}{2} \\ 0, &\rho-\tfrac{\alpha}{2}, &1-\tfrac{\delta}{2}
\end{matrix}\, \Big  \vert \, |x|^2 
\Big )
\]
for all $x$ such that $|x| \ne 1$.
\end{enumerate}
\end{corollary}

The above result resembles Corollary~\ref{cor:power}. There are, however, only partial analogues of Corollary~\ref{cor:DVF} for functions supported in the unit ball. We provide two results of this kind: Corollary~\ref{cor_unit_ball_Delta_zero} and Theorem~\ref{thm:DVP}.

\begin{corollary}\label{cor_unit_ball_Delta_zero}
Let $V(x)$ be a solid harmonic polynomial of degree $l\ge 0$ and $\delta = d + 2 l$. 
Assume that $-d < \alpha < 0$ or $\alpha > 0$, as well as $2 \rho > -d - l$ and $\sigma > -1$
. Define
\begin{equation}\label{prop2_f1}
f(x) := V(x) (1-|x|^2)_+^{\sigma} \,
{_2\tilde{F}_1}\Big(\,\begin{matrix} 1 + \sigma - \tfrac{\alpha}{2}, & \tfrac{\alpha}{2} - \rho \\ \multicolumn{2}{c}{1 + \sigma} \end{matrix}\, \Big \vert \, 1-|x|^2 \Big) .
\end{equation}
Then
\begin{equation}\label{prop2_f2}
(-\Delta)^{\alpha/2} f(x) = \frac{2^{\alpha} \Gamma(\rho + \tfrac{\delta}{2})}
{\Gamma(1 + \sigma - \tfrac{\alpha}{2})} \, V(x) |x|^{\alpha-2 \rho} {_2\tilde{F}_1}
\Big(\,\begin{matrix} \rho + \tfrac{\delta}{2}, & \tfrac{\alpha}{2} - \sigma \\[0.1em] \multicolumn{2}{c}{\rho + \tfrac{\delta - \alpha}{2}} \end{matrix}\, \Big \vert \, |x|^2 \Big)
\end{equation}
for all $x$ such that $0 < |x| < 1$, and also for $x = 0$ if $2 \rho > \alpha - l$.
\end{corollary}

With $\sigma = \tfrac{\alpha}{2}$, we simply have
\[
(-\Delta)^{\alpha/2} f(x) = \frac{2^{\alpha} \Gamma(\rho + \tfrac{\delta}{2})}
{\Gamma(\rho + \tfrac{\delta - \alpha}{2})} \, V(x) |x|^{\alpha-2 \rho} .
\]
In particular, setting $\alpha > 0$, $\sigma=\tfrac{\alpha}{2}$ and $2 \rho = \alpha - \delta$ in the above proposition we obtain a function
\begin{equation}\label{corr2_f1}
f(x):=V(x) (1-|x|^2)_+^{\alpha/2} \,
{_2\tilde{F}_1}\Big(\,\begin{matrix} 1, & \tfrac{\delta}{2} \\ \multicolumn{2}{c}{1+\tfrac{\alpha}{2}} \end{matrix}\, \Big \vert \, 1-|x|^2 \Big),
\end{equation}
which is continuous except at $x = 0$, is equal to zero in the complement of the unit ball and satisfies $(-\Delta)^{\alpha/2} f(x) = 0$ for all $x$ in the unit ball except at $x = 0$.

\begin{remark}
The Green function $G(x, y)$ for $(-\Delta)^{\alpha/2}$ in the unit ball was found in~\cite{bib:r38a}. Traditionally, it is written in the form that was first stated in~\cite{bib:bgr61},
\[
G(x, y) = \frac{\Gamma(\tfrac{d}{2})}{2^\alpha \pi^{d/2} (\Gamma(\tfrac{\alpha}{2}))^2} \, \frac{1}{|x - y|^{d - \alpha}} \int_0^{\tfrac{(1 - |x|^2) (1 - |y|^2)}{|x - y|^2}} \frac{s^{\alpha/2 - 1}}{(1 + s)^{d/2}} \, \d s .
\]
Using formula~3.194.1 in~\cite{bib:gr07}, we obtain
\[
G(x, y) = \frac{\Gamma(\tfrac{d}{2}) (1 - |x|^2)^{\alpha/2} (1 - |y|^2)^{\alpha/2}}{2^\alpha \pi^{d/2} \Gamma(\tfrac{\alpha}{2}) |x - y|^d} \, {_2\tilde{F}_1}\Big(\,\begin{matrix}\tfrac{d}{2}, & \tfrac{\alpha}{2} \\ \multicolumn{2}{c}{1 + \tfrac{\alpha}{2}} \end{matrix}\, \Big \vert \, {-\frac{(1 - |x|^2) (1 - |y|^2)}{|x - y|^2}} \Big) .
\]
For $y = 0$, by formula~9.131.1 in~\cite{bib:gr07},
\[
\begin{aligned}
 G(x, 0) & = \frac{\Gamma(\tfrac{d}{2}) (1 - |x|^2)^{\alpha/2}}{2^\alpha \pi^{d/2} \Gamma(\tfrac{\alpha}{2}) |x|^d} \, {_2\tilde{F}_1}\Big(\,\begin{matrix}\tfrac{d}{2}, & \tfrac{\alpha}{2} \\ \multicolumn{2}{c}{1 + \tfrac{\alpha}{2}} \end{matrix}\, \Big \vert \, 1 - \frac{1}{|x|^2} \Big) \\
 & = \frac{\Gamma(\tfrac{d}{2}) (1 - |x|^2)^{\alpha/2}}{2^\alpha \pi^{d/2} \Gamma(\tfrac{\alpha}{2})} \, {_2\tilde{F}_1}\Big(\,\begin{matrix}\tfrac{d}{2}, & 1 \\ \multicolumn{2}{c}{1 + \tfrac{\alpha}{2}} \end{matrix}\, \Big \vert \, 1 - |x|^2 \Big) = \frac{\Gamma(\tfrac{d}{2})}{2^\alpha \pi^{d/2} \Gamma(\tfrac{\alpha}{2})} \, f(x) ,
\end{aligned}
\]
where $f$ is given by~\eqref{corr2_f1} with $l = 0$ and $V(x) \equiv 1$. Thus Corollary~\ref{cor_unit_ball_Delta_zero} can be viewed as an alternative derivation of the expression for $G(x, 0) = C f(x)$ (the value of the constant $C$ can be found by comparing the asymptotic expansion of $f$ and the Riesz potential near $0$, we omit the details). We remark that the expression for $G(x, y)$ can be obtained from that for $G(x, 0)$ by means of the Kelvin transformation, see~\cite{bib:l72,bib:r38a}.
\end{remark}


For our last main result we need some additional notation. Solid harmonic polynomials of degree $l \ge 0$ form a finite-dimensional space, having dimension 
\[
M_{d,l} := \frac{d+2l-2}{d+l-2} \, \binom{d+l-2}{l}.
\] 
For each $l$ we fix a linear basis of this space, denoted by $V_{l,m}$ with $m = 1, \dots, M_{d,l}$, which is orthonormal with respect to the surface measure $\mu$ on the unit sphere. Since the space $L^2(\mu)$ is a direct sum of the above linear spaces, the collection $V_{l,m}$, with $l = 0, 1, \dots$ and $m = 1, \dots, M_{d,l}$, is an orthonormal basis of $L^2(\mu)$. \mk{For further properties of harmonic polynomials, see~\cite{bib:abw01,bib:dx14}.}

Recall that the Jacobi polynomials are defined as
\begin{equation}\label{eq:jacobiq}
\begin{aligned}
 P_n^{(\alpha,\beta)}(z) & := \frac{\Gamma(\alpha+1+n)}{n!} \,
 {_2\tilde{F}_1}\Big(\,\begin{matrix}
-n, & 1+\alpha+\beta+n \\ \multicolumn{2}{c}{\alpha+1} \end{matrix}\, \Big \vert \, \frac{1-z}{2} \Big) \\
& = \frac{(-1)^n \Gamma(\beta+1+n)}{n!} \,
 {_2\tilde{F}_1}\Big(\,\begin{matrix}
-n, & 1+\alpha+\beta+n \\ \multicolumn{2}{c}{\beta+1} \end{matrix}\, \Big \vert \, \frac{1+z}{2} \Big). 
\end{aligned}
\end{equation}
Given $d \ge 1$ and $\alpha > 0$, we denote
\begin{equation}\label{eq:jacobip}
 P_{l,m,n}(x) := V_{l,m}(x) P^{(\alpha/2,d/2+l-1)}_n(2 |x|^2 - 1),
\end{equation}
where $l, n \ge 0$ and $1 \le m \le M_{d,l}$. This system of polynomials forms a complete orthogonal system in \mk{$L^2(w)$}, the weighted $L^2$ space with weight function $w(x) = (1 - |x|^2)_+^{\alpha/2}$, see Proposition~2.3.1 in~\cite{bib:dx14}. Finally, we define
\[
 p_{l,m,n}(x) := (1 - |x|^2)_+^{\alpha/2} P_{l,m,n}(x) .
\]
The following proposition will play an important role in the study of eigenvalues of the fractional Laplace operator in~\cite{DKK_2015}. For the special case of this result when $d=\alpha=2$, see~\cite{bib:Gibson}.

\begin{theorem}\label{thm:DVP}
Assume that $\alpha > 0$, $l, n \ge 0$, $1 \le m \le M_{d,l}$. Then
\begin{equation}\label{eq:jacobibase}
(-\Delta)^{\alpha/2} p_{l,m,n}(x) = \frac{2^\alpha \Gamma(1 + \tfrac{\alpha}{2} + n) \Gamma(\tfrac{\delta + \alpha}{2} + n)}
{n! \, \Gamma(\tfrac{\delta}{2} + n)} \, P_{l,m,n}(x)
\end{equation}
for all $x$ such that $|x| < 1$. In other words, the polynomials $P_{l,m,n}$ form a complete orthogonal system of eigenfunctions of the operator \mk{$f \mapsto (-\Delta)^{\alpha}(w f)$} in \mk{$L^2(w)$,} the weighted $L^2$ space with weight function $w(x) = (1 - |x|^2)_+^{\alpha/2}$.
\end{theorem}

\section{Proofs}

%

\subsection{Technical results}

The following result essentially reduces the case of a general solid harmonic polynomial $V$ of degree $l \ge 0$ to $V(x) \equiv 1$ and $l = 0$.

\begin{proposition}\label{prop_Bochner}
Let $V(x)$ be a solid harmonic polynomial of degree $l\ge 0$.
Let $f$ and $\tilde f$ be two radial functions in $\r^d$ and $\r^{d+2l}$ with the same profile function.
\begin{enumerate}[label={\rm{(\roman*)}}]
\item\label{prop_Bochner_1}
If $0<\alpha<d$, $x\in \r^d \setminus \{0\}$ and the function $y\mapsto V(y)f(y)|x-y|^{\alpha-d}$ is in $L^1(\r^d)$, then 
\begin{equation}\label{eqn_Bochner_I_alpha}
\mk{(-\Delta)^{-\alpha/2}}(Vf)(x)=V(x) \mk{(-\Delta)^{-\alpha/2}}\tilde f(\tilde x),
\end{equation}
for any $\tilde x\in \r^{d+2l}$ such that $|\tilde x|=|x|$. \mk{Similar statement also holds for $x = 0$ if the function $y \mapsto (1 + |V(y)|) f(y) |y|^{\alpha-d}$ is in $L^1(\r^d)$.}
\item\label{prop_Bochner_2}
If $\alpha>0$, $x\in \r^d$,  $f(y)$ is smooth in a neighbourhood of $x$ and the function  $y\mapsto V(y)f(y)(1+|y|)^{-\alpha-d}$ is in $L^1(\r^d)$, then 
\begin{equation}\label{eqn_Bochner_Delta_alpha}
(-\Delta)^{\alpha/2}(Vf)(x)=V(x) (-\Delta)^{\alpha/2}\tilde f(\tilde x),
\end{equation}
for any $\tilde x\in \r^{d+2l}$ such that $|\tilde x|=|x|$. 
\end{enumerate}
\end{proposition}

\noindent
The proof is based on the following result.

\begin{bochners_relation}[{Corollary on page 72 in~\cite{Stein_book}}]
Let $f$ and $\tilde f$ be two radial Schwartz functions in $\r^d$ and $\r^{d+2l}$ with the same profile function, and $x \in \r^d$, $\tilde{x} \in \r^{d+2l}$, $|x| = |\tilde{x}|$. Then 
\begin{equation}\label{Bochner's relation}
\fourier_d(Vf)(x)=\i^l V(x) \fourier_{d+2l} \tilde f(\tilde x). 
\end{equation}
\end{bochners_relation}

\begin{proof}[Proof of Proposition~\ref{prop_Bochner}]
Everywhere in this proof we will assume that $x\in \r^d$ and $\tilde x \in \r^{d+2l}$, and $|x|=|\tilde x|$. Our first goal is to establish the following result: For any radial Schwartz functions $\phi$ and $\tilde \phi$ 
on $\r^d$ and $\r^{d+2l}$ with the same profile it is true that 
\begin{equation}\label{eqn_Bochners_convolution}
(\fourier_d^{-1} \phi) * (Vf)(x)=V(x) (\fourier_{d+2l}^{-1} \tilde \phi)*\tilde f(\tilde x).  
\end{equation}
To prove this result, \mk{for $\tilde x \in \r^{d+2l}$} we define
\[
\tilde h(\tilde x):=\tilde \phi(\tilde x) \fourier_{d+2l} \tilde f(\tilde x),
\]
and denote by $h(x)$ the radial function in $\r^d$ having the same profile as $\tilde h$. By Bochner's relation~\eqref{Bochner's relation} we have
\[
\phi(x) \fourier_d(V f)(x) = \i^l \phi(x) V(x) \fourier_{d+2l} \tilde{f}(\tilde{x}) = \i^l V(x) h(x) ,
\]
and so
\[
(\fourier_d^{-1} \phi)*(Vf) =\fourier^{-1}_d \bigl[ \phi  \fourier_d (Vf) \bigr] = \i^l \fourier^{-1}_d \bigl[ V h  \bigr] .
\]
On the other hand,
\[
(\fourier^{-1}_{d+2l} \tilde \phi)*\tilde f = \fourier^{-1}_{d+2l} \bigl[\tilde \phi \fourier_{d+2l} \tilde f\bigr] = \fourier^{-1}_{d+2l} \tilde h .
\]
Identity~\eqref{eqn_Bochners_convolution} follows from connection between Fourier and inverse Fourier transforms, and another application of Bochner's relation~\eqref{Bochner's relation},
\[
\i^l \fourier^{-1}_d \bigl[ V h  \bigr](x) = (2 \pi)^{-d} \i^l \overline{\fourier_d \bigl[ V h  \bigr](x)} = (2 \pi)^{-d} V(x) \overline{\fourier_{d+2l} \tilde h(\tilde x)} = V(x) \fourier^{-1}_{d+2l} \tilde h(\tilde x) .
\]
Now we are ready to prove part~(i) of Proposition~\ref{prop_Bochner}. Let us consider first the case when $f$ is a Schwartz function. We set $\phi\mk{_\varepsilon}$ and $\tilde \phi\mk{_\varepsilon}$ to be functions on $\r^d$ and $\r^{d+2l}$ with the same profile function 
\begin{align*}
\phi\mk{_\varepsilon}(x)&=\tilde \phi\mk{_\varepsilon}(\tilde x)=(\varepsilon^2+r^2)^{-\alpha/2}e^{-\varepsilon r^2}, & |x|=|\tilde x|=r. 
\end{align*}
Note that both $\phi\mk{_\varepsilon}$ and $\tilde \phi\mk{_\varepsilon}$ are radial Schwartz functions and their profile \mk{increases} to $r^{-\alpha}$ as $\varepsilon \to 0^+$. \mk{In particular, $\phi_\varepsilon$ and $\tilde \phi_\varepsilon$ converge in the space of Schwartz distributions to $\phi(x) = |x|^{-\alpha}$ and $\tilde \phi(x) = |\tilde x|^{-\alpha}$, respectively. It follows that $\fourier_d^{-1} \phi$ and $\fourier_{d+2l}^{-1} \tilde \phi$ also converge as Schwartz distributions to kernel functions of Riesz potential operators $(-\Delta)^{-\alpha/2}$ in $\r^d$ and $\r^{d+2l}$, respectively (see~\eqref{def_I_alpha} and~\eqref{def_I_alpha_Fourier}).} Using~\eqref{eqn_Bochners_convolution} \mk{for $\phi_\varepsilon$ and $\tilde \phi_\varepsilon$ and passing to the limit $\varepsilon \to 0^+$,} we conclude that   
\eqref{eqn_Bochner_I_alpha} holds true \mk{when $f$ is a Schwartz function}. 

The general case (when $f$ is not a Schwartz function, but satisfies the integrability condition stated in part~\ref{prop_Bochner_1}) \mk{follows by approximation. Indeed, let $f_n$ be a sequence of radial Schwartz functions on $\r^d$ which converges to $f$ in the weighted $L^1(\r^d)$ norm, with weight function $y \mapsto |y|^l |x - y|^{\alpha - d}$ (or $y \mapsto 1 + |y|^l$ when $x = 0$). Then it is easy to see that the corresponding radial functions $\tilde f_n$ on $\r^{d+2l}$ converge to $\tilde f$ in the weighted $L^1(\r^{d + 2 l})$ norm with weight function $\tilde{y} \mapsto |\tilde{x} - \tilde{y}|^{\alpha - d}$, and~\eqref{eqn_Bochner_I_alpha} follows.} 

The proof of part~\ref{prop_Bochner_2} is obtained in the same way, except that \mk{we use} the functions  
\begin{align*}
\phi\mk{_\varepsilon}(x)& = \tilde \phi\mk{_\varepsilon}(\tilde x)=(\varepsilon^2+r^2)^{\alpha/2}e^{-\varepsilon r^2}, && |x|=|\tilde x|=r. 
\end{align*}
The details are left to the reader. 
\end{proof}

We often prove our results for a restricted range of parameters, and then extend them by analytic continuation. This is possible thanks to \mk{our next proposition}.

\begin{proposition}
\label{prop:morera}
Let $U$ be an open set in $\c^n$, $n \ge 1$. Assume that $z\mapsto f_z(x)$ is an analytic function on $U$ for almost all $x\in \r^d$. 
\begin{enumerate}[label={\rm{(\roman*)}}]
\item\label{prop:morera:1}
Let us fix $x \in \r^d$ and $0 < \alpha < d$. If 
\[
y \mapsto |y-x|^{-d+\alpha} \sup_{z\in U} |f_z(y)| \quad \text{ is in $L^1(\r^d)$,}
\]
then $\mk{(-\Delta)^{-\alpha/2}}f_z(x)$ is an analytic function of $z\in U$. 
\item\label{prop:morera:2}
Let us fix $x \in \r^d$ and $\alpha_0>0$.  
If 
\begin{equation}\label{eq:morera:condition}
y \mapsto (1+|y|)^{-d-\alpha_0} \sup_{z\in U} |f_z(y)| \quad \text{ is in $L^1(\r^d)$}
\end{equation}
and the function $y \mapsto f_z(y)$ is smooth in some small ball $B(x,r)$ with partial derivatives (of arbitrary order) bounded uniformly in $(y,z)$ on $B(x,r) \times U$, then 
$(-\Delta)^{\alpha/2}f_z(x)$ is an analytic function of $z \in U$ and of $\alpha$ in the half-plane $\re(\alpha)>\alpha_0$.
\end{enumerate}
\end{proposition}

\noindent
The proof is based on the following result.

\begin{lemma}[Lemma 1.31 in \cite{Samko_book}]\label{lemma131} Let $D$ be an open set in $\c$, $\Omega \subseteq \r^d$ and $g(y,z)$ be an analytic function in $z\in D$ for almost all $y\in \Omega$. If $y \mapsto \sup_{z \in D} |g(y,z)|$ is in $L^1(\Omega)$ then the function 
$
z\mapsto \int_{\Omega} g(y,z) \d y
$
is analytic on $D$.
\end{lemma}

\begin{proof}[Proof of Proposition~\ref{prop:morera}]
Let us prove part~\ref{prop:morera:1}. Let $z=(z_1,z_2,\dots,z_n) \in U$. 
Lemma \ref{lemma131} guarantees that the function $\mk{(-\Delta)^{-\alpha/2}}f_z(x)$ is analytic in each variable $z_i$, $1\le i \le n$. Applying Hartog's Theorem we obtain joint analyticity in $z \in U$.  

For the proof of part~\ref{prop:morera:2}, we observe that the boundedness of partial derivatives of the functions $f_z$ implies that for $y \in B(0,r)$
\[
|\Delta_y^k f_z(x)| \le C |y|^k, 
\]
for some constant $C=C(k,r)$ which does not depend on $z$. 
Therefore, the function
\[
y \mapsto \sup_{\substack{z \in U \\ \alpha_0 < \re \alpha < k}} \bigl| |y|^{-d-\alpha} \Delta^k_y f_z(x) \bigr|
\]
is integrable in $B(0,r)$. On the other hand, the same function is integrable in $\r^d \setminus B(0, r)$ by condition~\eqref{eq:morera:condition}. The desired result under additional condition $\re \alpha < k$ follows now from~\eqref{def_fractional_Laplacian} and Lemma~\ref{lemma131} (and Hartog's Theorem). Since $k$ was arbitrary, the proof is complete.
\end{proof}

\subsection{Main results}

\begin{proof}[Proof of Proposition~\ref{prop_riesz_mellin}]
When $l = 0$ and $V(x) \equiv 1$, the result follows rather easily from Fubini's theorem and the semigroup property of Riesz potential operators $(-\Delta)^{-(\alpha+\beta)/2} = (-\Delta)^{-\alpha/2} (-\Delta)^{-\beta/2}$. More precisely, by formula~(1.1.12) in~\cite{bib:l72} (see also formula~(8) on page 118 in~\cite{Stein_book}),
\begin{equation}
\label{eqn_riesz_semigroup}
\int_{\r^d} \frac{1}{|x - y|^{d - \alpha}} \, \frac{1}{|y|^{d - \beta}} \, \d y = \frac{\gamma_d(\alpha) \gamma_d(\beta)}{\gamma_d(\alpha + \beta)} \, \frac{1}{|x|^{d - \alpha - \beta}}
\end{equation}
for all $x \in \r^d$ and all $\alpha, \beta \in \c$ such that $\re \alpha, \re \beta > 0$ and $\re \alpha + \re \beta < d$. Applying Fubini's Theorem, using~\eqref{eqn_riesz_semigroup} for $\alpha \in (0, d)$ and $\beta = d - 2s$, and substituting \mk{$\tilde{s} + \tfrac{\alpha}{2}$} for $s$, we obtain
\begin{align*}
(-\Delta)^{-\alpha/2} f(x) & = \frac{1}{2 \pi \i} \int_{\lambda + \i \r} \mellin \phi(s) \left( \frac{1}{\gamma_d(\alpha)} \int_{\r^d} \frac{1}{|x - y|^{2s}} \, \frac{1}{|y|^{d - \alpha}} \, \d y \right) \d s \\
& = \frac{1}{2 \pi \i} \int_{\lambda + \i \r} \mellin \phi(s) \, \frac{\gamma_d(d-2s)}{\gamma_d(d+\alpha-2s)} \, \frac{1}{|x|^{2s-\alpha}} \, \d s \\
& = \frac{1}{2 \pi \i} \int_{\lambda - \alpha/2 + \i \r} \mellin \phi(\tilde{s} + \tfrac{\alpha}{2}) \, \frac{\gamma_d(d-\alpha-2\tilde{s})}{\gamma_d(d-2\tilde{s})} \, \frac{1}{|x|^{2\tilde{s}}} \, \d \tilde{s} ,
\end{align*}
which is equivalent to the desired result~\eqref{eqn_riesz_mellin}. The use of Fubini's Theorem is justified since for $s = \lambda + \i t$, the integral
\begin{align*}
\frac{1}{\gamma_d(\alpha)} \int_{\r^d} \left| \frac{1}{|x - y|^{2s}} \, \frac{1}{|y|^{d - \alpha}} \right| \d y & = \frac{1}{\gamma_d(\alpha)} \int_{\r^d} \frac{1}{|x - y|^{2\lambda}} \, \frac{1}{|y|^{d - \alpha}} \, \d y
\end{align*}
is finite and does not depend on $s$, while the integral in~\eqref{eqn_inverse_mellin} is absolutely convergent.

The case of general $V$ (that is, $l \ge 0$) reduces to $V(x) \equiv 1$ in dimension $\delta = d + 2 l$ by Proposition~\ref{prop_Bochner}.
\end{proof}

\begin{proof}[Proof of Theorem \ref{thm:IVG}]
We provide a detailed argument when Condition~A is satisfied. Suppose first that the parameters $p$, $q$, $m$, $n$, $\mathbf{a}$ and $\mathbf{b}$ satisfy Conditions~S and~A. In this case the desired result is a simple consequence of Proposition~\ref{prop_riesz_mellin}. Indeed, for $\lambda \in (-\underline{b}, 1 - \overline{a}) \cap (\tfrac{\alpha + l}{2}, \tfrac{d + l}{2})$ (such a number $\lambda$ exists by Condition~S and assumption~\eqref{eq:IVG:cond}) we have
\begin{align*}
\mk{(-\Delta)^{-\alpha/2}} f(x) & = V(x) \, \frac{1}{2\pi \i} \int_{\lambda - \alpha/2 + \i \r} \cG^{mn}_{pq} \Big( \,\begin{matrix} \mathbf{a} \\ \mathbf{b} \end{matrix}\, \Big\vert \, s + \tfrac{\alpha}{2} \Big ) \frac{\Gamma(s)\Gamma(\tfrac{\delta - \alpha}{2}-s)}{2^\alpha \Gamma(\tfrac{\alpha}{2}+s)\Gamma(\tfrac{\delta}{2}-s)} \, |x|^{-2s} \d s \\
& = 2^{-\alpha} V(x) G^{m+1,n+1}_{p+2,q+2}\Big( \,\begin{matrix} 1-\tfrac{\delta - \alpha}{2}, &\mathbf{a} + \tfrac{\alpha}{2}, & 0 \\ 0, &\mathbf{b} + \tfrac{\alpha}{2}, &1-\tfrac{\delta}{2} \end{matrix}\, \Big\vert \, |x|^2 \Big),
\end{align*}
where the second equality is a consequence of the definition~\eqref{eq:MG} of $\cG^{mn}_{pq}$ and the definition~\eqref{def_Meijer_G} of the Meijer G-function.

Relaxing Condition~S is possible by analytic continuation and Proposition~\ref{prop:morera}. \mk{Indeed,} let $U$ be a region with compact closure contained in the set of admissible parameters of $f$ under Condition~A. Using the definition~\eqref{def_Meijer_G}, with an appropriately modified contour of integration in the general case, one can prove that the constants in the asymptotic estimates~\eqref{eq:G:asymp} for the Meijer G-function can be chosen continuously with respect to the parameters; we omit the details. It follows that the supremum of $|f(y)|$ taken over all parameters from $U$ \mk{has, for some $\varepsilon > 0$, the following properties: (i) it} is locally bounded in $y \ne 0$\mk{; (ii)} it is $O(|y|^{-d + \varepsilon})$ as $y \to 0$\mk{; (iii) it is} $O(|y|^{-\alpha - \varepsilon})$ as $|y| \to \infty$. Hence, the assumptions of Proposition~\ref{prop:morera}\ref{prop:morera:1} are satisfied, and the desired result follows by analyticity of both sides of~\eqref{eq:DVG:I}.

\mk{We now sketch the proof in the remaining cases. When Conditions~S and~B are satisfied and $\nu > 1$, the argument is very similar, based on Proposition~\ref{prop_riesz_mellin}. In order to relax the condition on $\nu$ in the analytic continuation argument, one also needs the following regularity of the Meijer G-function near $1$: if $0 \le \nu \le 1$ and $\varepsilon > 0$, then
\begin{equation}\label{Meijer_G_asymptotics_1}
\begin{aligned}
 G^{mn}_{pq}\Big( 
\,\begin{matrix}
\mathbf{a} \\ \mathbf{b}
\end{matrix}\, \Big  \vert \, r 
\Big )= 
O(|r - 1|^{\nu - \varepsilon - 1}) && \text{as $r\to 1$,}
\end{aligned}
\end{equation}
again with constant in the asymptotic notation depending continuously on the parameters. This can be proved using the estimates for the hypergeometric function ${_p\tilde{F}_{q-1}}$ and~\eqref{eq:G:F}; we omit the details.}

\mk{Similarly, when Conditions~S and either~C or~D are satisfied, and in addition $-\underline{b} < \nu < 1 - \overline{a}$, then $\underline{\lambda} < \overline{\lambda}$ and Proposition~\ref{prop_riesz_mellin} applies. The conditions on $\mathbf{a}$ and $\mathbf{b}$ again can be relaxed by analytic continuation.}
\end{proof}

\begin{proof}[Proof of Theorem~\ref{thm:DVG}]
Assume first that $x \ne 0$ and the parameters $\alpha$, $p$, $q$, $m$, $n$, $\mathbf{a}$ and $\mathbf{b}$ satisfy Condition~A and in addition
\begin{equation}
\label{eq:DVG:cond_prime}
\begin{aligned}
0 < \alpha & < d, &\qquad 2(1 - \overline{a}) & > l , &\qquad -2 \underline{b} & < d + l - \alpha .
\end{aligned}
\end{equation}
Let us denote the function in the right-hand side of~\eqref{eq:DVG:D} by $g(x)$. Then $g(y)$ is continuous in $y \ne 0$, and by~\eqref{eq:DVG:cond_prime} and~\eqref{eq:G:asymp}, for some $\varepsilon > 0$ we have $|g(y)| = O(|y|^{-d + \varepsilon})$ as $y \to 0^+$ and $|g(y)| = O(|y|^{-\alpha - \varepsilon})$ as $y \to \infty$ (because $l + 2 \underline{b} - \alpha > -d$ and $l - 2 (1 - \overline{a}) - \alpha < -\alpha$). Conditions~\eqref{eq:DVG:cond_prime} ensure also that $2 (1 - \overline{a} + \tfrac{\alpha}{2}) > \alpha + l$ and $-2(\underline{b} - \tfrac{\alpha}{2}) < d + l$. It follows that we can compute $(-\Delta)^{-\alpha/2} g$ via Theorem~\ref{thm:IVG}: performing this computation and using reduction formulas~\eqref{reduction1} and~\eqref{reduction2} we check that:
\[
\begin{aligned}
(-\Delta)^{-\alpha/2} g(x) & = V(x) G^{m+2,n+2}_{p+4,q+4}\Big( 
\,\begin{matrix}
1-\tfrac{\delta - \alpha}{2}, & 1 - \tfrac{\delta}{2}, & \mathbf{a}, & 0, & \tfrac{\alpha}{2} \\ 0, & \tfrac{\alpha}{2} , &\mathbf{b}, &1-\tfrac{\delta - \alpha}{2}, & 1 - \tfrac{\delta}{2}
\end{matrix}\, \Big  \vert \, |x|^2 
\Big ) \\
& = V(x) G^{mn}_{pq}\Big( 
\,\begin{matrix}
\mathbf{a} \\ \mathbf{b}
\end{matrix}\, \Big  \vert \, |x|^2 
\Big ) = f(x) .
\end{aligned}
\]
To finish the proof we only need to note that $g$ is continuous at $x$ and apply Proposition~\ref{prop_Delta_inverse_I_alpha}. If $g$ is continuous at $0$, we also have~\eqref{eq:DVG:D} at $x = 0$.

Relaxing the additional conditions is done using Proposition~\ref{prop:morera} in a very similar way as in the proof of Theorem~\ref{thm:IVG}. Both $f$ and $g$ are analytic functions of $\alpha$, $\mathbf{a}$ and $\mathbf{b}$ in the admissible region specified by condition~\eqref{eq:DVG:cond}. From the estimates~\eqref{eq:G:asymp} of the Meijer G-function it follows that $f$ satisfies the assumptions of Proposition~\ref{prop:morera}\ref{prop:morera:2} for every domain $U$ with compact closure in the admissible region.

\mk{The remaining cases, when Conditions~B, C or D are satisfied, are very similar, and we only sketch the argument. Under Condition~B, one initially assumes that $\nu > 1 + \alpha$ and then uses also~\eqref{Meijer_G_asymptotics_1} for the analytic continuation. If Conditions~C or~D hold, the analytic continuation part requires no modifications, but in order to use Theorem~\ref{thm:IVG}, one needs to initially assume that either $1 + 2 \tfrac{\nu - 1 - \alpha}{q - p} > l$ when Condition~C is satisfied, or $1 - 2 \tfrac{\nu - 1 - \alpha}{p - q} < d + l - \alpha$ when Condition~D holds. We omit the details.}
\end{proof}

\subsection{Full space}

\begin{proof}[Proof of Corollary~\ref{cor:power}]
The result follows immediately from Theorem~\ref{thm:IVG} (when $-d < \alpha < 0$) or Theorem~\ref{thm:DVG} (when $\alpha > 0$) by using~\eqref{eq:power:G}.
\end{proof}

\begin{proof}[Proof of Corollary~\ref{cor:DVF}]
Denote $\mathbf{b} = (b_1, \dots, b_{q-1}, \tfrac{\delta}{2})$. Suppose first that none of $a_j$ is a non-positive integer. Using~\eqref{eq:FG}, Theorem~\ref{thm:IVG} (when $-d < \alpha < 0$) or Theorem~\ref{thm:DVG} (when $\alpha > 0$) and then the reduction formulas~\eqref{reduction1} and \eqref{reduction2}, we obtain the desired result. The restriction that none of $a_j$ is a non-positive integer can be relaxed by analytic continuation using Proposition~\ref{prop:morera}.
\end{proof}


\subsection{Unit ball and its complement}

\begin{proof}[Proof of Corollary~\ref{cor:powerball}]
The result follows immediately from Theorem~\ref{thm:IVG} (when $-d < \alpha < 0$) or Theorem~\ref{thm:DVG} (when $\alpha > 0$) by using~\eqref{eq:powerball:G} and~\eqref{eq:powerballc:G}.
\end{proof}

\begin{proof}[Proof of Corollary~\ref{cor_unit_ball_Delta_zero}]
From formula~8.4.49.22 in~\cite{Prudnikov_V3} we find \mk{that} 
\[
\begin{aligned}
f(x) & = G^{20}_{22}\Big( 
\,\begin{matrix}
\tfrac{\alpha}{2}, & 1 + \rho + \sigma - \tfrac{\alpha}{2} \\ 0, & \rho
\end{matrix}\, \Big  \vert \, |x|^2 
\Big ).
\end{aligned}
\]
Applying Theorem~\ref{thm:IVG} (when $-d < \alpha < 0$) or Theorem~\ref{thm:DVG} (when $\alpha > 0$), and simplifying the result we obtain
\begin{align*}
(-\Delta)^{\alpha/2} f(x)&=
2^{\alpha}V(x)
G^{31}_{44}\Big( 
\,\begin{matrix}
1-\tfrac{\delta + \alpha}{2}, & 0, & 1 + \rho + \sigma - \alpha, & -\tfrac{\alpha}{2} \\ 
0, & -\tfrac{\alpha}{2}, & \rho - \tfrac{\alpha}{2}, & 1-\tfrac{\delta}{2}
\end{matrix}\, \Big  \vert \, |x|^2 
\Big )\\
&=
2^{\alpha}V(x)
G^{11}_{22}\Big( 
\,\begin{matrix}
1-\tfrac{\delta + \alpha}{2}, & 1 + \rho + \sigma - \alpha \\  
\rho - \tfrac{\alpha}{2}, & 1-\tfrac{\delta}{2}
\end{matrix}\, \Big  \vert \, |x|^2 
\Big ),
\end{align*}
where in the second step we have used reduction formula~\eqref{reduction2}. Formula~\eqref{prop2_f2} follows 
from the above expression and~\eqref{eq:G:F}.
\end{proof}


\mk{In the next result, we will use the following property of Meijer G-function, which follows easily from the definition~\eqref{def_Meijer_G} as the inverse Mellin transform: if $k$ is an integer, $\mathbf{a}' = (a_2, \dots, a_p)$ and $\mathbf{b}' = (b_1, \dots, b_{q-1})$, then
\begin{equation}
\label{Meijer_G_swap}
G^{mn}_{pq}\Big( 
\,\begin{matrix}
c, & \mathbf{a}' \\ \mathbf{b}', & c+k
\end{matrix}\, \Big  \vert \, r
\Big )=(-1)^k G^{m+1,n-1}_{p,q}\Big( 
\,\begin{matrix}
\mathbf{a}', & c \\ c+k, & \mathbf{b}'
\end{matrix}\, \Big  \vert \, r
\Big ).
\end{equation}}

\begin{proof}[Proof of Theorem~\ref{thm:DVP}]
Define $\delta := d + 2 l$, $V(x) := V_{l,m}(x)$,
\[
f(x) := \frac{n! \, p_{l,m,n}(x)}{\Gamma(1 + \tfrac{\alpha}{2} + n)} = V(x) (1-|x|^2)_+^{\alpha/2}
{_2\tilde{F}_1}\Big(\,\begin{matrix} -n, & \tfrac{\delta + \alpha}{2} + n \\ \multicolumn{2}{c}{1 + \tfrac{\alpha}{2}} \end{matrix}\, \Big \vert \, 1-|x|^2 \Big) 
\]
(the equality follows from the definitions of $p_{l,m,n}$ and $P_{l,m,n}$ and the first expression for the Jacobi polynomial in~\eqref{eq:jacobiq}) and
\[
g(x) := \frac{n! \, P_{l,m,n}(x)}{\Gamma(\tfrac{\delta}{2} + n)} = (-1)^n V(x)
{_2\tilde{F}_1}\Big(\,\begin{matrix} \tfrac{\delta + \alpha}{2} + n , & -n \\ \multicolumn{2}{c}{\tfrac{\delta}{2}} \end{matrix}\, \Big \vert \, |x|^2 \Big)
\]
(here we used the second expression for the Jacobi polynomial in~\eqref{eq:jacobiq}). By formula~8.4.49.22 in~\cite{Prudnikov_V3} and~\eqref{Meijer_G_swap}, we have
\[
\begin{aligned}
f(x) & = V(x) G^{20}_{22}\Big( 
\,\begin{matrix}
1 + \tfrac{\alpha}{2} + n, & 1 - \tfrac{\delta}{2} - n \\ 1 - \tfrac{\delta}{2}, & 0
\end{matrix}\, \Big  \vert \, |x|^2 
\Big ) \\
& = (-1)^n V(x) G^{11}_{22}\Big( 
\,\begin{matrix}
1 - \tfrac{\delta}{2} - n, & 1 + \tfrac{\alpha}{2} + n \\ 0, & 1 - \tfrac{\delta}{2}
\end{matrix}\, \Big  \vert \, |x|^2 
\Big ) .
\end{aligned}
\]
Using Theorem~\ref{thm:DVG} and the reduction formula~\eqref{reduction2} we find that when $|x| \ne 1$,
\[
(-\Delta)^{\alpha/2} f(x) =
(-1)^n 2^{\alpha}V(x)
G^{11}_{22}\Big( 
\,\begin{matrix}
1 - \tfrac{\delta + \alpha}{2} - n, & 1 + n , \\ 
0, & 1-\tfrac{\delta}{2}
\end{matrix}\, \Big  \vert \, |x|^2 
\Big ) .
\]
By~\eqref{eq:G:F}, when $|x| < 1$,
\[
\begin{aligned}
(-\Delta)^{\alpha/2} f(x) & = \frac{(-1)^n 2^\alpha \Gamma(\tfrac{\delta + \alpha}{2} + n)}{n!} \, V(x)
{_2\tilde{F}_1}\Big(\,\begin{matrix} \tfrac{\delta + \alpha}{2} + n , & -n \\ \multicolumn{2}{c}{\tfrac{\delta}{2}} \end{matrix}\, \Big \vert \, |x|^2 \Big) \\
& = \frac{2^\alpha \Gamma(\tfrac{\delta + \alpha}{2} + n)}{n!} \, g(x) ,
\end{aligned}
\]
as desired.
\end{proof}

\end{document}